\documentclass{amsart}

\usepackage{amssymb}    
\usepackage{amsmath}    
\usepackage{amsthm}     
\usepackage{amssymb}
\usepackage{mathrsfs}

\usepackage[margin=1in]{geometry}
\numberwithin{equation}{section}

\newtheorem{theorem}{Theorem}[section]
\newtheorem{lemma}[theorem]{Lemma}
\newtheorem{proposition}[theorem]{Proposition}
\newtheorem{corollary}[theorem]{Corollary}

\newtheorem{conjecture}[theorem]{Conjecture}

\newenvironment{prf}[1]{\trivlist
\item[\hskip
\labelsep{\it #1.\hspace*{.3em}}]}{
\endtrivlist}

\newtheorem{predefinition}[theorem]{Definition}

\newtheorem{preremark}[theorem]{Remark}
\newenvironment{remark}{\begin{preremark}\rm}{\end{preremark}}
\newtheorem{prenotation}[theorem]{Notation}
\newenvironment{notation}{\begin{prenotation}\rm}{\end{prenotation}}
\newtheorem{preexample}[theorem]{Example}
\newenvironment{example}{\begin{preexample}\rm}{\end{preexample}}
\newtheorem{preclaim}[theorem]{Claim}

\newtheorem{prequestion}[theorem]{Question}

 \makeatletter
\def\emppsubsection{\@startsection{subsection}{2}{\z@}{-3.25ex plus -1ex minus -.2ex}{-1em}{\bf}}

\makeatother

\newcommand \ZZ {{\mathbb Z}}
\newcommand \QQ {{\mathbb Q}}
\newcommand \NN {{\mathbb N}}
\newcommand  \FF {{\mathbb F}}
\newcommand \GG {{\mathbb G}}
\newcommand \dime {\mathop{\rm dim}}

\newcommand \Hom {\mathop{\rm Hom}}

\newcommand \Hdr {\mathop {H^1_{{\rm dR}}(\mathcal{S}_m)}}

\newcommand \EE {{\mathbb E}}

\usepackage{color}

\begin{document}

\title{The de Rham cohomology of the Suzuki curves}
\author{Beth Malmskog}
\address{Beth Malmskog\\
Department of Mathematics and Computer Science\\
Colorado College\\
Colorado Springs, CO 80903, USA}
\email{beth.malmskog@gmail.com}

\author{Rachel Pries}
\address{Rachel Pries, Department of Mathematics, 
Colorado State University, 
Fort Collins, CO 80523, USA}
\email{pries@math.colostate.edu}

\author{Colin Weir}
\address{Colin Weir,
The Tutte Institute for Mathematics and Computing\\
Ottawa, Ontario, Canada}

\email{colinoftheweirs@gmail.com}

\subjclass[2010]{Primary: 11G10, 11G20, 14F40, 14H40, 20C20.  Secondary: 14L15,
20C33}

%
%
%

 \keywords{Suzuki curve, Suzuki group, Ekedahl-Oort type, de Rham cohomology, Dieudonn\'e module, modular representation}
 
\begin{abstract}
For a natural number $m$, let $\mathcal{S}_m/\FF_2$ be the $m$th Suzuki curve.  
We study the mod $2$ Dieudonn\'{e} module of $\mathcal{S}_m$, which gives the equivalent information 
as the Ekedahl-Oort type or the structure of the $2$-torsion group scheme of its Jacobian.  
We accomplish this by studying the de Rham cohomology of $\mathcal{S}_m$.
For all $m$, we determine the structure of the de Rham cohomology as a 
$2$-modular representation of the $m$th Suzuki group
and the structure of a submodule of the mod $2$ Dieudonn\'{e} module.
For $m=1$ and $2$, we determine the complete structure of the mod $2$ Dieudonn\'{e} module.  
\end{abstract}

 \maketitle
\section{Introduction}

The structure of the de Rham cohomology of the Hermitian curves 
as a representation of ${\rm PGU}(3,q)$ was studied in \cite{dum95, dum99, hj90}.
The mod $p$ Dieudonn\'e module and the Ekedahl-Oort type of the 
Hermitian curves were determined in \cite{PW12}.
In this paper, we study the analogous structures for the Suzuki curves.  

For $m \in \NN$, let $q_0=2^m$, and let $q=2^{2m+1}$.
The Suzuki curve $\mathcal{S}_m$ is the smooth projective connected curve over $\FF_2$ given by the affine equation:
\[z^q+z=y^{q_0}(y^q+y).\]
It has genus $g_m=q_0(q-1)$.

The number of points of $\mathcal{S}_m$ over $\FF_q$ is $\#\mathcal{S}_m\left(\FF_q\right)=q^2+1$; 
which is {\it optimal} in that it reaches Serre's improvement to the Hasse-Weil bound 
\cite[Proposition 2.1]{HansenStich}.
In fact, $\mathcal{S}_m$ is the unique $\mathbb{F}_q$-optimal curve of genus $g_m$ \cite{FTunique}.
Because of the large number of rational points relative to their genus, 
the Suzuki curves provide good examples of Goppa codes
\cite{GK08},
\cite{GKT}, \cite{HansenStich}. 

The automorphism group of $\mathcal{S}_m$ is the Suzuki group ${\rm Sz}(q)$, 
whose order $q^2(q-1)(q^2+1)$ is very large compared with $g_m$.
The Suzuki curve $\mathcal{S}_m$ is the Deligne-Lusztig curve associated with the group ${\rm Sz}(q)= {}^2B_2(q)$  \cite[Proposition 4.3]{H92}.

The $L$-polynomial of $\mathcal{S}_m/\FF_q$ is $(1+\sqrt{2q} t +q t^2)^{g_m}$
and so $\mathcal{S}_m$ is supersingular for each $m \in {\mathbb N}$ \cite[Proposition~4.3]{H92}.
This implies that the Jacobian ${\rm Jac}(\mathcal{S}_m)$ is isogenous over $\bar{{\mathbb F}}_2$ to a product of 
supersingular elliptic curves.
In particular, ${\rm Jac}(\mathcal{S}_m)$ has $2$-rank $0$; it
has no points of order $2$ over $\overline{\FF}_2$.  

The $2$-torsion group scheme ${\rm Jac}(\mathcal{S}_m)[2]$ is a ${\rm BT}_1$-group scheme of rank $2^{2g_m}$.
In \cite{FGMPW}, the authors show that the $a$-number of ${\rm Jac}(\mathcal{S}_m)[2]$ is 
$a_m=q_0(q_0+1)(2q_0+1)/6$; in particular, ${\rm lim}_{m \to \infty} a_m/g_m=1/6$.
However, the Ekedahl-Oort type of ${\rm Jac}(\mathcal{S}_m)[2]$ is not known.  
Understanding the Ekedahl-Oort type is equivalent to 
understanding the structure of the de Rham cohomology or the mod $2$ reduction of the Dieudonn\'e module
as a module under the actions of the operators Frobenius $F$ and Verschiebung $V$.

In this paper, we study the de Rham cohomology group $\Hdr$ of the Suzuki curves.
The $2$-modular representations of the Suzuki group are understood from \cite{martineau, ChastofskyFeit, sin, Liu}.
Using results about the cohomology of Deligne-Lusztig varieties from \cite{Lusztig} and \cite{Gross},
we determine the multiplicity of each irreducible $2$-modular 
representation of ${\rm Sz}(q)$ in $\Hdr$ in Corollary~\ref{Theoremrep}.

Let $D_m$ denote the mod $2$ reduction of the Dieudonn\'{e} module of (the Jacobian of) $\mathcal{S}_m$. 
It is an $\EE$-module where $\EE$ is the non-commutative ring generated over $\bar{\FF}_2$ by $F$ and $V$ with the 
relations $FV=VF=0$.
As explained in Section \ref{Sresults}, there is an $\EE$-module
decomposition $D_m=D_{m, 0} \oplus D_{m, \not = 0}$, where
the $\EE$-submodule $D_{m,0}$ is the trivial eigenspace for the action of an automorphism $\tau$ of order $q-1$.

In Proposition \ref{PEOsuztrivial}, we determine the structure of $D_{m,0}$ completely 
by finding that its Ekedahl-Oort type is $[0,1,1,2,2, \ldots, q_0-1,q_0]$.
This yields the following corollary.

\begin{corollary} (Corollary~\ref{maincor})
If $2^m \equiv 2^e \bmod 2^{e+1} +1$, then the ${\mathbb E}$-module ${\mathbb E}/\EE(V^{e+1} + F^{e+1})$
occurs as an ${\mathbb E}$-submodule of the mod 2 Dieudonn\'e module $D_m$ of ${\mathcal S}_m$.
In particular, 
\begin{enumerate}
\item ${\mathbb E}/\EE(V^{m+1} + F^{m+1})$ occurs as an ${\mathbb E}$-submodule of $D_m$ for all $m$;

\item ${\mathbb E}/\EE(V+F)$ occurs as an ${\mathbb E}$-submodule of $D_m$ if $m$ is even; and

\item ${\mathbb E}/\EE(V^2+F^2)$ occurs as an ${\mathbb E}$-submodule of $D_m$ if $m \equiv 1 \bmod 4$.
\end{enumerate}
\end{corollary}

We have less information about $D_{m, \not = 0}$, the sum of the non-trivial eigenspaces for $\tau$.
In Section \ref{Snontrivial}, we explain a connection between the Ekedahl-Oort type and irreducible subrepresentations of $\Hdr$.
This motivates Conjecture \ref{Ceasyword}, in which we conjecture that the $\EE$-module 
$\EE/\EE(V^{2m+1} + F^{2m+1})$ occurs with multiplicity $4^m$ in $D_m$.

We determine the complete structure of the mod 2 Dieudonn\'{e} module 
$D_m$ for $m=1$ and $m=2$ in Propositions \ref{m=1}-\ref{m=2}.
To do this, we explicitly compute a basis for $\Hdr$ for all $m \in \NN$ in Section ~\ref{Sexplicitbasis}
and, for $m=1,2$, we compute the actions of $F$ and $V$ on this basis. 

There is a similar result in \cite{duursmaskabelund} for the first Ree curve, which is defined over $\FF_3$, namely the authors determine its mod $3$ Dieudonn\'{e} module.

Malmskog was partially supported by NSA grant H98230-16-1-0300.
Pries was partially supported by NSF grant DMS-15-02227. 
We would like to thank Jeff Achter for helpful comments.

\subsection{Notation}  \label{Snotation}

We begin by establishing some notation regarding $p$-torsion group schemes, 
mod $p$ Dieudonn\'e modules, and Ekedahl-Oort types,
taken directly from \cite[Section 2]{PW12}.

Let $k$ be an algebraically closed field of characteristic $p >0$. 
Suppose $A$ is a principally polarized abelian variety of dimension $g$ defined over $k$.
Consider the multiplication-by-$p$ morphism $[p]:A \to A$ which is a finite flat morphism of degree $p^{2g}$.
It factors as $[p]=V \circ F$.  Here $F:A \to A^{(p)}$ is the relative Frobenius morphism 
coming from the $p$-power map on the structure sheaf; it is purely inseparable of degree $p^g$.  
The Verschiebung morphism  $V:A^{(p)} \to A$ is the dual of $F_{A^{\rm dual}}$.  

The {\it $p$-torsion group scheme} of $A$, denoted $A[p]$, is the kernel of $[p]$.  
It is a finite commutative group scheme annihilated by $p$, again having morphisms $F$ and $V$, 
with ${\rm Ker}(F)={\rm Im}(V)$ and ${\rm Ker}(V)={\rm Im}(F)$.
The principal polarization of $A$ induces a symmetry on $A[p]$ as defined in \cite[5.1]{O:strat};
when $p = 2$, there are complications with the polarization which are resolved in \cite[9.2, 9.5, 12.2]{O:strat}.

There are two important invariants of (the $p$-torsion of) $A$: the $p$-rank and $a$-number.
The {\it $p$-rank} of $A$ is $f=\dime_{\FF_p} \Hom(\mu_p, A[p])$
where $\mu_p$ is the kernel of Frobenius on $\GG_m$.
Then $p^f$ is the cardinality of $A[p](k)$.
The {\it $a$-number} of $A$ is $a=\dime_k \Hom(\alpha_p, A[p])$ 
where $\alpha_p$ is the kernel of Frobenius on $\GG_a$.

One can describe the group scheme $A[p]$ using 
the {\it mod $p$ Dieudonn\'e module}, i.e., the modulo $p$ reduction of the 
covariant Dieudonn\'e module, see e.g., \cite[15.3]{O:strat}.
More precisely, there is an equivalence of categories 
between finite commutative group schemes over $k$ annihilated by $p$ 
and left $\EE$-modules of finite dimension.
Here ${\mathbb E} = k[F, V ]$ denotes the non-commutative ring generated by 
semi-linear operators $F$ and $V$ with the relations $F V = V F = 0$ and $F \lambda = \lambda^pF$ and $\lambda V = V \lambda^p$ for all $\lambda \in k$. 
Let ${\mathbb E}(A_1, \ldots)$ denote the left ideal of ${\mathbb E}$ generated by $A_1, \ldots$.

Furthermore, there is a bijection between isomorphism classes of $2g$ dimensional left $\EE$-modules and {\it Ekedahl-Oort types}. 
To find the Ekedahl-Oort type, let $N$ be the mod $p$ Dieudonn\'e module of $A[p]$.  
The canonical filtration of $N$
is the smallest filtration of $N$ stabilized by the action of $F^{-1}$ and $V$; denote it by 
\[0 = N_0 \subset N_1 \subset \cdots N_z=N.\]
The canonical filtration can be extended to a final filtration; the Ekedahl-Oort type is the tuple
$[\nu_1, \ldots, \nu_g]$, where the $\nu_i$ are the dimensions of the images of $V$
on the subspaces in the final filtration.

For example,  
let $I_{t,1}$ denote the $p$-torsion group scheme of rank $p^{2t}$ having $p$-rank $0$ and $a$-number $1$. 
Then $I_{t,1}$ has Dieudonn\'e module ${\mathbb E}/{\mathbb E}(F^{t} + V^{t})$ 
and Ekedahl-Oort type $[0,1, \ldots, t-1]$ \cite[Lemma 3.1]{Prgroupscheme}.

For a smooth projective curve $X$, by \cite[Section 5]{Oda69}, there is an isomorphism
of ${\mathbb E}$-modules between the contravariant mod $p$ Dieudonn\'e module 
of the $p$-torsion group scheme
${\rm Jac}(X)[p]$ and the de Rham cohomology $H^1_{\rm dR}(X)$.\footnote{ 
Differences between the covariant and contravariant theory do not cause a problem in this paper since all objects we
consider are symmetric.}

In the rest of the paper, $p=2$.  

\section{The de Rham cohomology as a representation for the Suzuki group}

In this section, we analyze the de Rham cohomology $\Hdr$ of the Suzuki curve
as a $2$-modular representation of the Suzuki group.

\subsection{Some ordinary representations}

Suzuki determined the irreducible ordinary characters and representations of ${\rm Sz}(q)$ \cite{Suzuki}.  
Consider the following four unipotent representations of ${\rm Sz}(q)$.  
Let $W_S$ denote the Steinberg representation of dimension $q^2$.
Let $W_0$ be the trivial representation of dimension 1.  
Let $W_+$ and $W_-$ be the two unipotent cuspidal representations of ${\rm Sz}(q)$,
associated to the two ordinary characters of ${\rm Sz}(q)$ of degree $q_0(q-1)$ \cite{Suzuki}. 
Then $W_+$ and $W_-$ each have dimension $q_0(q-1)$.  

In \cite[Theorem 6.1]{Lusztig}, Lusztig studied the compactly supported $\ell$-adic cohomology of the 
affine Deligne-Lusztig curves.
For the Suzuki curves, he proved that the ordinary representations 
$W_S$, $W_+$, $W_-$, $W_0$ are the eigenspaces under Frobenius
and that each appears with multiplicity 1.  

\subsection{Modular representations of the Suzuki group}

The absolutely irreducible $2$-modular representations of ${\rm Sz}(q)$ are well-understood 
\cite{martineau, ChastofskyFeit, sin, Liu}.

Let $q=2^{2m+1}$.
We recall some results about the $2$-modular representations of the Suzuki group ${\rm Sz}(q)$ from \cite{martineau}.
Fix a generator $\zeta$ of $\mathbb{F}^*_q$.
Let $\theta \in {\rm Aut}(\FF_q)$ be such that $\theta^2(\alpha)=\alpha^2$ for all $\alpha \in \FF_q$, i.e., 
$\theta$ is the square root of Frobenius.

The Suzuki group acts on $\mathcal{S}_m$.
Let $\tau \in \text{Sz}(q)$ be an element of order $q-1$; without loss of generality, 
we suppose that $\tau$ acts on $\mathcal{S}_m$ by
\[\tau : y \mapsto \zeta y,  \ z \mapsto \zeta^{2^m+1}z.\]

Then $\text{Sz}(q)$ has an irreducible 4-dimensional 2-modular representation $V_0$ in which $\tau\mapsto M$, where $M\in {\rm GL}_4(\mathbb{F}_q)$ is the matrix
\[M=\left(\begin{array}{cccc}\zeta^{\theta+1} & 0 & 0 & 0 \\0 & \zeta & 0 & 0 \\0 & 0 & \zeta^{-1} & 0 \\0 & 0 & 0 & \zeta^{-(\theta+1)}\end{array}\right).\]

%

For $0 \leq i \leq 2m$, consider the automorphism $\alpha_i$ of ${\rm Sz}(q)$ 
induced by the automorphism $x \mapsto x^{2^i}$ of $\FF_q$.
Let $V_i$ be the $4$-dimensional $\FF_q {\rm Sz}(q)$-module where $g \in {\rm Sz}(q)$ acts as $g^{\alpha_i}$ on $V_0$.

Let $I$ be a subset of $N=\ZZ/(2m+1)\ZZ$.  Define $V_I=\otimes_{j \in I} V_j$, with $V_\emptyset$ being the trivial module.
Then $V_I$ is an absolutely irreducible $2$-modular representation of ${\rm Sz}(q)$.
By \cite[Lemma 1]{martineau}, if $I \neq J$ then $V_I$ and $V_J$ are geometrically non-isomorphic and 
$\{V_I \mid I \subset N\}$ is the complete set of simple $\overline{\FF}_2 {\rm Sz}(q)$-modules.
Note that $V_I$ has dimension $4^{|I|}$ and that $V_N$ is the Steinberg module.  

By \cite[Theorem, page 1]{sin}, for $I, J \subset N$, there are no non-trivial extensions of $V_I$ by $V_J$, 
namely ${\rm Ext}^1_{\bar{\FF}_2 {\rm Sz}(q)}(V_I, V_J)=0$.

The Frobenius $x \mapsto x^2$ on $\FF_q$ acts on $\{V_i \}$ taking $V_i \mapsto V_{i+1 \bmod 2m+1}$.
Note that $\oplus_{I \in {\mathcal I}} V_I$ is an $\FF_2 {\rm Sz}(q)$-module if and only if ${\mathcal I}$ is invariant under Frobenius or, equivalently, if and only if 
$\{I \mid I \in {\mathcal I}\}$ is invariant under the translation $i \mapsto i+1 \bmod 2m+1$.

For $i\in N$, let $\phi_i$ denote the Brauer character associated to the $4$-dimensional module $V_i$.  For $I\subseteq N$, let $\phi_I=\prod_{i\in I} \phi_i$, so $\phi_I$ is the character associated to the module $V_I$.  Then $\{\phi_I:I\subseteq N\}$ is a complete set of Brauer characters for ${\rm Sz}(q)$.  

By \cite[Theorem 3.4]{ChastofskyFeit}, $\phi_i^2=4+2\phi_{i+m+1}+\phi_{i+1}$.
Using this relation, 
Liu constructs a graph with vertex set $N$ and edge set $\{(i,i+1), (i,i+1+m):i\in N\}$.  Edges of the form $(i,i+1)$ are called short edges and edges of the form $(i, i+1+m)$ are called long edges.  Two vertices $i,j$ are called adjacent if they are connected by a long edge, i.e., if $i-j\equiv\pm m\bmod{2m+1}$.  A set $I^{\prime}\subseteq N$ is called circular if no vertices of $I=N\setminus{I^{\prime}}$ are adjacent.  A set $I\subseteq S$ is called good if $I^{\prime}=N\setminus{I}$ is circular.

The decompositions of $W_+$ and $W_-$ into irreducible $2$-modular representations are known. 

\begin{theorem} \label{TLiu}
Liu \cite[Theorem 3.4]{Liu} 
The irreducible $2$-modular representation $V_I$ appears in $W_{\pm}$ if and only if $I$ is good, i.e., 
if and only if there do not exist $i, j \in I$ such that $j-i\equiv \pm m \bmod 2m+1$.  
In this case, the multiplicity of $V_I$ in $W_\pm$ is $2^{m-|I|}$.
\end{theorem}

\subsection{Modular representation of the de Rham cohomology}

The de Rham cohomology $H^1_{{\rm dR}}(\mathcal{S}_m)$ is an $\mathbb{F}_2[{\rm Sz}(q)]$-module
of dimension $2g_m=2q_0(q-1)$.  
We consider the decomposition of $H^1_{{\rm dR}}(\mathcal{S}_m)$
into irreducible 2-modular representations of the Suzuki group ${\rm Sz}(q)$.  

\begin{corollary}\label{Theoremrep}
The irreducible $2$-modular representation $V_I$ appears in $\Hdr$ if and only if there do not exist $i, j \in I$ such that $j-i\equiv \pm m \bmod 2m+1$.
If $V_I$ appears in $\Hdr$ then its multiplicity is $2^{m+1-|I|}$.
Thus the $2$-modular ${\rm Sz}(q)$-representation of $\Hdr$ is:
\begin{equation}\label{repforHdr}
\Hdr \simeq \bigoplus_{I \ {\rm good}} V_I^{2^{m+1-|I|}}.
\end{equation}
\end{corollary}

\begin{proof}
In \cite[page 2535]{Gross}, Gross uses \cite[Theorem 6.1]{Lusztig} to prove that, as a ${\rm Sz}(q)$-representation,
the $\ell$-adic cohomology of the smooth projective curve ${\mathcal S}_m$ is:
\begin{equation*} \label{E4eigenspace}
H^1(\mathcal{S}_{m, \bar{\FF}_2}, \bar{\QQ}_\ell)  \simeq W_+ \oplus W_-.
\end{equation*}
By \cite[Theorem 2]{katzmessing}, the characters of $H^1(\mathcal{S}_{m, \bar{\FF}_2}, \bar{\QQ}_\ell)$ and 
$H^1_{\rm crys}(\mathcal{S}_m, {\rm Frac}(W(\bar{\FF}_2)))$ as representations of ${\rm Sz}(q)$ are the same, 
and thus the representations are isomorphic.
The de Rham cohomology is the reduction modulo $2$ of the crystalline cohomology.
Thus the result follows from Theorem~\ref{TLiu}.
\end{proof}

\begin{example} When $m=1$, then $\Hdr \simeq (V_0\oplus V_1 \oplus V_2)^2 \oplus V_\emptyset^4$.
\end{example}

\begin{example} When $m=2$, then
\[\Hdr \simeq
\left(V_{\{0,1\}}\oplus V_{\{1,2\}} \oplus V_{\{2,3\}} \oplus V_{\{3,4\}}\oplus V_{\{4,0\}}\right)^2 \oplus
(V_0\oplus V_1 \oplus V_2 \oplus V_3\oplus V_4)^4 \oplus V_\emptyset^8.\]
\end{example}

\begin{remark} For $m \leq 10$, we verified Corollary~\ref{Theoremrep} 
using the multiplicity of the eigenvalues for $\tau$ on $\Hdr$. 
\end{remark}

\section{The Dieudonn\'e module and de Rham cohomology}

In this section, we study the structure of the mod $2$ Dieudonn\'e module $D_m$ of the Suzuki curve ${\mathcal S}_m$ or, equivalently, the structure of $\Hdr$ as an $\EE$-module.

\subsection{Results and conjectures} \label{Sresults}


The chosen element $\tau \in {\rm Sz}(q)$ of order $q-1$ acts on the mod $2$ Dieudonn\'e module $D_m$.
Let $D_{m,0}$ denote the trivial eigenspace and $D_{m, \not = 0}$ denote the direct sum of the non-trivial eigenspaces.
Since $F$ and $V$ commute with $\tau$, 
they stabilize $D_{m, 0}$ and $D_{m, \not = 0}$; thus there is an 
$\EE$-module decomposition $D_m = D_{m, 0} \oplus D_{m, \not =0}$.

In Section \ref{Strivial}, we prove the next proposition; it determines the $\EE$-module structure of $D_{m,0}$.

\begin{proposition} \label{PEOsuztrivial}
Let $m \in \NN$ and let $q_0 = 2^m$.  The trivial eigenspace $D_{m,0}$ of the mod $2$ Dieudonn\'e module 
of the Suzuki curve has Ekedahl-Oort type $[0,1,1,2,2, \ldots, q_0-1,q_0]$;
in particular, it has rank $2q_0$, $2$-rank $0$, and $a$-number $2^{m-1}$.
\end{proposition}

We have less information about the $\EE$-module structure of $D_{m, \not =0}$.
In Section \ref{Snontrivial}, we explain how the non-trivial representations $V_I$ in $\Hdr$
lead to $\EE$-submodules $D_I$ of the mod $2$ Dieudonn\'e module of $\Hdr$.
We would like to understand how to determine the $\EE$-module structure of $D_I$ from the representation $V_I$ 
for the subset $I \subset N=\ZZ/(2m+1)\ZZ$.  
In Section \ref{Snontrivial}, we consider a particular representation $W_m$, 
and make the following conjecture.


\begin{conjecture} \label{Ceasyword}
The multiplicity of $\EE/\EE(F^{2m+1}+V^{2m+1})$ in the mod $2$ 
Dieudonn\'e module $D_m$ of $\mathcal{S}_m$ is $4^m$. 
\end{conjecture}

We verify Conjecture \ref{Ceasyword} for $m=1$ and $m=2$ in Propositions \ref{m=1} and \ref{m=2}.
In fact, for $m=1$ and $m=2$, we determine the mod $2$ Dieudonn\'e module $D_m$ completely. 
To do this, we find a basis for $H^1_{\rm dR}({\mathcal S}_m)$ for all $m$ in Section \ref{Sexplicitbasis}. 
For $m=1$, we explicitly compute the action of $F$ and $V$ on this basis, proving that:

\begin{proposition} \label{m=1}
When $m=1$, then 
the mod $2$ Dieudonn\'e module of ${\mathcal S}_1$ is
\[D_1=(\EE/\EE(F^3+V^3))^4 \oplus \EE/\EE(F^2+V^2).\]
\end{proposition}

For $m=2$,  we determine the action of $F$ and $V$ on $H^1_{\rm dR}({\mathcal S}_m)$ using 
Magma \cite{Magma}.
Consider the $\EE$-module $\EE(Z)$ generated by $X_1, X_2, X_3$ with the following relations:
$V^3 X_1 - F^3 X_2=0$; $V^4X_2 - F^3 X_3=0$; and $V^3 X_3 - F^4 X_1=0$. 
Then $\EE(Z)$ is symmetric and has rank $20$, $p$-rank $0$, and $a$-number $3$.

\begin{proposition} \label{m=2}
When $m=2$, then the mod $2$ Dieudonn\'e module of ${\mathcal S}_2$ is 
\[D_2=\left(\EE/\EE(F^5+V^5)\right)^{16} \oplus \left(\EE(Z)\right)^4 \oplus (\EE/\EE(F^3+V^3) \oplus \EE/\EE(F+V)).\]
\end{proposition}

\subsection{The trivial eigenspace} \label{Strivial}

The eigenspace $D_{m,0}$ is the subspace of $H^1_{{\rm dR}}(\mathcal{S}_m)$ 
of elements fixed by $\tau$. 
Since $\tau$ acts fixed point freely on the $4$-dimensional module $V_i$ for each $i$ \cite[proof of Lemma 3]{martineau}, 
the generators of $H_{{\rm dR}}^1(\mathcal{S}_m)$ which are fixed by $\tau$ are 
exactly those in $V_I$ for $I=\emptyset$.
In other words, the representation for $D_{m,0}$ consists of the $2^{m+1}=2q_0$ 
copies of the trivial representation in \eqref{repforHdr}.  

\begin{proof} (Proof of Proposition \ref{PEOsuztrivial})
Let $\mathcal{C}_{m,0}$ be the quotient curve of $\mathcal{S}_m$ by the subgroup $\langle \tau \rangle$.  
Then $\mathcal{C}_{m,0}$ is a hyperelliptic curve of genus $q_0$ by \cite[Theorem 6.9]{GKT}.  

The de Rham cohomology $H^1_{{\rm dR}}(\mathcal{C}_{m,0})$ of $\mathcal{C}_{m,0}$ is isomorphic as an 
${\mathbb E}$-module to $D_{m,0}$.  
Thus the trivial eigenspace $D_{m,0}$ for the mod $2$ Dieudonn\'e module of $\mathcal{S}_m$
is isomorphic to the mod $2$ Dieudonn\'e module of $C_{m,0}$; in particular, it has rank $2q_0$.

Since ${\mathcal S}_m$ has $2$-rank $0$, so does $\mathcal{C}_{m,0}$.  
Thus $C_{m,0}$ is a hyperelliptic curve of $2$-rank $0$.
By \cite[Corollary~5.3]{elkinpries}, the Ekedahl-Oort type of $C_{m,0}$ is $[0,1,1,2,2, \ldots, q_0-1, q_0]$;
this implies that the $a$-number is $2^{m-1}$.
\end{proof}
  
We determine the $\EE$-module structure of $D_{m,0}$ by applying results from \cite[Section 5]{elkinpries}.  

\begin{proposition}  \label{Pdieu} \cite[Proposition 5.8]{elkinpries}
The mod 2 Dieudonn\'e module $D_{m,0}$ is the $\mathbb{E}$-module generated 
as a $k$-vector space by $\{X_1, \ldots, X_{q_0}, Y_1, \ldots, Y_{q_0}\}$ 
with the actions of $F$ and $V$ given by:
\begin{enumerate}
\item
$F(Y_j) = 0$.

\item
$V(Y_j) = 
\begin{cases}
Y_{2j} & \text{if $j \leq {q_0}/2$,}\\
0 & \text{if $j > {q_0}/2$.}
\end{cases}$

\item
$F(X_i)=
\begin{cases}
X_{j/2} & \text{if $j$ is even,}\\
Y_{q_0-(j-1)/2} & \text{if $j$ is odd.}
\end{cases}$

\item 
$V(X_j)=
\begin{cases}
0 & \text{if $j \leq (q_0-1)/2$,}\\
-Y_{2q_0-2j+1} & \text{if $j > (q_0-1)/2$.}
\end{cases}$
\end{enumerate}
\end{proposition}

We have an explicit description of the generators and relations of $D_{m,0}$ as follows.

\begin{notation} \label{Nbij} \cite[Notation 5.9]{elkinpries} 
Fix $c=q_0 \in \NN$. 
Consider the set $I=\{j \in \NN \mid \lceil (c+1)/2 \rceil \leq j \leq c\}$, which has cardinality $\lfloor (c+1)/2 \rfloor$.
For $j \in I$, let $\ell(j)$ be the odd part of $j$ and let $e(j) \in {\mathbb Z}^{\geq 0}$ be such that $j=2^{e(j)}\ell(j)$.
Let $s(j)=c-(\ell(j)-1)/2$.  
Then $\{s(j) \mid j \in I\}=I$.
Also, let $m(j)=2c-2j+1$ and let $\epsilon(j) \in {\mathbb Z}^{\geq 0}$ be such that $t(j):=2^{\epsilon(j)}m(j) \in I$.  
Then $\{t(j) \mid j \in I\}=I$.
Thus, there is a unique bijection $\iota:I \to I$ such that $t(\iota(j))=s(j)$ for each $j \in I$.
\end{notation}

\begin{proposition} \label{Palg} \cite[Proposition 5.10]{elkinpries}
The set $\{X_j \mid j \in I\}$ generates the mod 2 Dieudonn\'e module $D_{m,0}$ as an ${\mathbb E}$-module
subject to the relations: $F^{e(j)+1}(X_j) + V^{\epsilon(\iota(j))+1}(X_{\iota(j)})$ for $j \in I$.
\end{proposition}

\begin{example}
\begin{enumerate}
\item When $m=1$ and the Ekedahl-Oort type is $[0,1]$, then
$D_{1,0} \simeq \EE/\EE(F^2+V^2)$ (group scheme $I_{2,1}$).

\item When $m=2$ and the Ekedahl-Oort type is $[0,1,1,2]$, then 
$D_{2,0} \simeq \EE/\EE(F+V) \oplus \EE/\EE(F^3+V^3)$ (group scheme $I_{1,1}\oplus I_{3,1}$).


%
\end{enumerate}
\end{example}

In the next result, we determine some of the ${\mathbb E}$-submodules of 
$D_{m,0}$ for general $m$.

\begin{proposition}\label{W0structure}
The ${\mathbb E}$-module ${\mathbb E}/\EE(V^{e+1} + F^{e+1})$
occurs as an ${\mathbb E}$-submodule of $D_{m,0}$ if and only if 
$2^m \equiv 2^e \bmod 2^{e+1} +1$.
In particular:
\begin{enumerate}
\item ${\mathbb E}/\EE(V^{m+1} + F^{m+1})$ occurs for all $m$;

\item ${\mathbb E}/\EE(V+F)$ occurs if and only if $m$ is even; and

\item ${\mathbb E}/\EE(V^2+F^2)=0$ occurs if and only if $m \equiv 1 \bmod 4$.
\end{enumerate}
\end{proposition}

\begin{proof}
Let $e \in {\mathbb Z}^{\geq 0}$.
By Proposition \ref{Palg}, 
the relation $(V^{e+1} + F^{e+1})X_j=0$ is only possible if $j=2^e \ell$ where $\ell$ is odd.
Write $s(j)=c-(\ell-1)/2$.  Then $F^{e+1}(X_j)=F(X_\ell)=Y_{s(j)}$.
Now $V(X_j)=-Y_{m(j)}$ where $m(j)=2c-2j+1$.
Also $V^{e+1}(X_j)=2^em(j)$.
Thus we need $s(j) = 2^e m(j)$.
This is equivalent to $2^{e+1} c - (j - 2^e) = 2^{2e+1}(2c-2j+1)$, which is equivalent to 
\[j=\frac{c2^{e+1}(2^{e+1}-1) + 2^e(2^{2e+1}-1)}{2^{2e+2}-1} = \frac{c 2^{e+1} + 2^e}{2^{e+1} + 1}.\]
This value of $j$ is integral if and only if $c \equiv 2^e \bmod 2^{e+1} + 1$.
Thus, the relation $(V^{e+1} + F^{e+1})X_j=0$ occurs if and only if 
$2^m \equiv 2^e \bmod 2^{e+1} +1$ and $j=(2^{e+1} q_0 + 2^e)/(2^{e+1} +1)$.
In particular, one checks that:
\begin{enumerate}
\item $(V^{m+1} + F^{m+1})X_{2^m}=0$;
\item the relation $(V+F)X_j=0$ occurs if and only if $m$ is even and $j=(2 \cdot 2^{m} +1)/3$;
\item the relation $(V^2+F^2)X_j=0$ occurs if and only if $m \equiv 1 \bmod 4$ and $j=(4 \cdot 2^m+2)/5$.
\end{enumerate}
\end{proof}

As a corollary, we determine cases when the $\EE$-module ${\mathbb E}/\EE(V^{e+1} + F^{e+1})$ appears in $D_m$.

\begin{corollary} \label{maincor}
If $2^m \equiv 2^e \bmod 2^{e+1} +1$, then the ${\mathbb E}$-module ${\mathbb E}/\EE(V^{e+1} + F^{e+1})$
occurs as an ${\mathbb E}$-submodule of the mod $2$ Dieudonn\'e module $D_m$ of ${\mathcal S}_m$.
In particular, 
\begin{enumerate}
\item ${\mathbb E}/\EE(V^{m+1} + F^{m+1})$ occurs as an ${\mathbb E}$-submodule of $D_m$ for all $m$;

\item ${\mathbb E}/\EE(V+F)$ occurs as an ${\mathbb E}$-submodule of $D_m$ if $m$ is even; and

\item ${\mathbb E}/\EE(V^2+F^2)$ occurs as an ${\mathbb E}$-submodule of $D_m$ if $m \equiv 1 \bmod 4$.
\end{enumerate}
\end{corollary}

\begin{proof}
By Proposition~\ref{W0structure}, ${\mathbb E}/\EE(V^{e+1} + F^{e+1})$
occurs as an ${\mathbb E}$-submodule of the mod $2$ Dieudonn\'e module $D_{m,0}$.
The result follows since $D_{m,0}$ is an $\EE$-submodule of $D_m$.
\end{proof}

\subsection{The nontrivial eigenspaces} \label{Snontrivial}



Recall that $D_{m, \not = 0}$ is the direct sum of the non-trivial eigenspaces for $\tau$.
Consider the canonical filtration of $D_{m, \not =0}$, which is
the smallest filtration stabilized under the action of $F^{-1}$ and $V$; denote it by 
\[0 = N_0 \subset N_1 \subset \cdots N_t=N.\]

By \cite[Chapter 2]{O:strat} (see also \cite[Section 2.2]{duursmaskabelund}), 
the blocks $B_i=N_{i+1}/N_{i}$ in the canonical filtration are representations for $\Hdr$.
On each block $B_i$, either (i) $V|_{B_i}=0$ in which case $B_i \subset {\rm Im}(F)$ and 
$F^{-1}: B_i^{(p)} \to B_j$ is an isomorphism to another block with index $j > i$; 
or (ii) $V: B_i^{(p)} \to B_j$ is an isomorphism to another block with index $j < i$.  
This action of $V$ and $F^{-1}$ yields a permutation $\pi$ of the set of blocks $B_i$.
Cycles in the permutation are in bijection with orbits ${\mathcal O}$ of the blocks under the action of $V$ and $F^{-1}$.

Fix an orbit ${\mathcal O}$ of a block $B_i$ under the action of $F^{-1}$ and $V$.
As in \cite[Section 5.2]{PW12}, this yields a word $w$ in $F^{-1}$ and $V$. 
From this, we produce a symmetric $\EE$-module $\EE(w)$ whose dimension over $k$ is the length of $w$.
Then $\EE(w)$ is an isotypic component of $D_{m,0}$.
The multiplicity of $\EE(w)$ in $D_{m, \not = 0}$ is the dimension of the block $B_i$ in ${\mathcal O}$.

By Corollary \ref{Theoremrep}, the representations occurring in $\Hdr$ are the representations in $W_\pm$,
namely the representations $V_I$ for $I$ a good subset of $N=\ZZ/(2m+1)\ZZ$.  

We now explain the motivation for Conjecture \ref{Ceasyword}.
Let $I_m=\{0, \ldots, m-1\}$.  The smallest power of $F$ that stabilizes $I_m$ is $2m+1$.
Consider the $2$-modular representation of $\rm{Sz}(q)$ given by $W_m=\oplus_{i=0}^{2m} F^i(V_{I_m})$.
For example, $W_1=V_0 \oplus V_1 \oplus V_2$ and
\[W_2=(V_0 \otimes V_1) \oplus (V_1 \otimes V_2) \oplus (V_2 \otimes V_3) \oplus (V_3 \otimes V_4) \oplus (V_4 \otimes V_0).\]

By definition, $W_m$ is an $\FF_2 {\rm Sz}(q)$-module of dimension $(2m+1)4^m$.
By Corollary~\ref{Theoremrep}, the $2$-modular representation $W_m$ appears with multiplicity $2$ in $\Hdr$.
Consider the $\EE$-module $\EE/\EE(F^{2m+1}+V^{2m+1})$; it has dimension $2(2m+1)$ over $k$.

The idea behind Conjecture \ref{Ceasyword} is that the subrepresentation $W_m^2$ of $\Hdr$
should correspond to a submodule of $D_{m}$ with structure $(\EE/\EE(F^{2m+1}+V^{2m+1}))^{4^m}$.
More precisely, Conjecture \ref{Ceasyword} would follow from the claims
that there is a unique $i$ such that $V_{I_m}$ is a subrepresentation of $B_i$, 
that $B_i$ is irreducible and thus equal to $V_{I_m}$, 
and that the word $w$ on the orbit of $B_i$ is $(F^{-1})^{2m+1}V^{2m+1}$.


\section{An Explicit Basis for the de Rham cohomology}\label{Sexplicitbasis}

In this section, we compute an explicit basis for $H^1_{{\rm dR}}(\mathcal{S}_m)$ for all $m$.
This material is needed to determine the mod 2 Dieudonn\'e module of $\mathcal{S}_m$ when $m=1$ and $m=2$ 
in Propositions \ref{m=1} and \ref{m=2}.
We determine the action of $F$ and $V$ on the basis elements explicitly here when $m=1$ and using Magma \cite{Magma}
when $m=2$. 

\subsection{Preliminaries}

Consider the affine equation
$z^q+z=y^{q_0}(y^q+y)$ for $\mathcal{S}_m$. 
Let $P_\infty$ be the point at infinity on ${\mathcal S}_m$.
Let $P_{(y,z)}$ denote the point $(y,z)$ on ${\mathcal S}_m$.

Define the functions $h_1,h_2\in\mathbb{F}_2\left(\mathcal{S}_m\right)$ by:
\[h_1:=z^{2q_0} + y^{2q_0+1}, \ h_2:=z^{2q_0}y+h_1^{2q_0}.\]

\begin{lemma} \label{valuations}
\begin{enumerate}
\item The function $y$ has divisor
\[{\rm div}(y)=\sum_{z\in\mathbb{F}_{q}} P_{(0,z)} -qP_{\infty}.\]

\item The function $z$ has divisor
\[{\rm div}(z)=\sum_{y\in\mathbb{F}_{q}^{\times} }P_{(y,0)} +(q_0+1)P_{(0,0)}-(q+q_0)P_{\infty}.\]

\item Let $S=\left\{(y,z)\in\mathbb{F}_q^2: y^{2q_0+1}=z^{2q_0}, (y,z)\neq(0,0)\right\}$.  The function $h_1$ has divisor
\[{\rm div}(h_1)=\sum_{(y,z)\in S}P_{(y,z)} +(2q_0+1)P_{(0,0)}-(q+2q_0)P_{\infty}.\]
\item The function $h_2$ has divisor
\[{\rm div}(h_2)=(q+2q_0+1)(P_{(0,0)}-P_{\infty}).\]
\end{enumerate}
\end{lemma}

\begin{proof}
The pole orders of these functions are determined in \cite[Proposition 1.3]{HansenStich}.  The orders of the zeros can be determined using the equation for the curve and the definitions of $h_1$ and $h_2$.
\end{proof}

Let $\mathcal{E}_m$ be the set of $(a,b,c,d) \subset \mathbb{Z}^4$ satisfying
\[\begin{array}{c}
0 \leq a, \quad 0 \leq b \leq 1,\quad 0 \leq c \leq q_0-1,\quad 0 \leq d \leq q_0-1,\\
 aq+b(q+q_0)+c(q+2q_0)+d(q+2q_0+1) \leq 2g-2.
\end{array}\]

\begin{lemma} \label{basisofHomega}
The following set is a basis of $H^0(\mathcal{S}_m, \Omega^1)$:
\[
\mathcal{B}_m:=\left\{g_{a,b,c,d}:=y^a z^b h_1^c h_2^d \, dy \, \mid \, (a,b,c,d) \in \mathcal{E}_m \right\}.
\]
\end{lemma}
\begin{proof}
See \cite[Proposition 3.7]{FGMPW}.
\end{proof}

A basis for $H^1(\mathcal{S}_m,\mathcal{O})$ can be built similarly.
Define the map
\[\pi:\mathcal{S}_m \rightarrow \mathbb{P}^1_y, \ 
(y,z) \mapsto y, \ \ P_{\infty} \mapsto\infty_y.\]
Let $0_y$ be the point on $\mathbb{P}^1_y$ with $y=0$.  
Then $\pi^{-1}(0_y) = \left\{(0,z):z\in\mathbb{F}_q\right\}$ has cardinality $q$.

\begin{lemma} \label{basisofHO}
The following set represents a basis of $H^1(\mathcal{S}_m,\mathcal{O})$:
\[
\mathcal{A}_m:=\left\{f_{a,b,c,d}:=\frac{1}{y^a z^b h_1^c h_2^d} \frac{z h_1^{q_0-1} h_2^{q_0-1}}{y}\, \mid \,(a,b,c,d) \in\mathcal{E}_m\right\}.
\]
\end{lemma}

\begin{proof}
Let $U_{\infty}=\mathcal{S}_m\setminus\pi^{-1}(\infty_y)=\mathcal{S}_m\setminus P_{\infty}$ and $U_0=\mathcal{S}_m\setminus\pi^{-1}(0_y)$.  
The elements of $H^1(\mathcal{S}_m,\mathcal{O})$ can be represented by classes of
functions that are regular on $U_{\infty} \cap U_0$, but are not regular on $U_{\infty}$ or regular on $U_0$.
In other words, these functions have a pole at $P_{\infty}$ and at some point in $\pi^{-1}(0_y)$.

Let $f=f_{a,b,c,d}$ for some $(a,b,c,d)\in\mathcal{E}_m$.  Then $f$ has poles only in $\left\{P_{\infty}, \pi^{-1}(0_y)\right\}$ by Lemma~\ref{valuations}. Let $Q=(0,\alpha)$ for some $\alpha\in\mathbb{F}_q^{\times}$.  Then
$v_{Q}(f)=-(a+1)\leq-1$.  Also,
\begin{eqnarray*}
v_{P_{\infty}}(f)&=&(a+1)(q)-(1-b)(q+q_0)-(q_0-1-c)(q+2q_0)-(q_0-1-d)(q+2q_0+1)\\
&=&aq+b(q+q_0)+c(q+2q_0)+d(q+2q_0+1) + (2q_0-2q_0q+1)\\
&\leq&2g_m-2+ (2q_0-2q_0q+1)\\
&=&(2q_0q-2q_0-2)+(2q_0-2q_0q+1)\\
&=&-1.
\end{eqnarray*}

So $f$ is regular on $U_{\infty}\cap U_0$ but not on $U_{\infty}$ or $U_0$.
By a calculation similar to \cite[Proposition 3.7]{FGMPW}, 
the elements of $\mathcal{A}_m$ are independent because each element has a different pole order at $P_{\infty}$.
The cardinality of $\mathcal{A}_m$ is $g_m={\rm dim}(H^1(\mathcal{S}_m,\mathcal{O}))$.  Thus $\mathcal{A}$ is a basis for $H^1(\mathcal{S}_m,\mathcal{O})$.

\end{proof}

\subsection{Constructing the de Rham cohomology}

 Let $\mathcal{U}$ be the open cover of $\mathcal{S}_m$ given by $U_{\infty}$ and $U_{0}$ 
 from the proof of Lemma \ref{basisofHO}.
 For a sheaf $\mathcal{F}$ on $\mathcal{S}_m$, let
\begin{eqnarray*}
C^0(\mathcal{U},\mathcal{F})&:=& \left\{g=(g_{\infty},g_0) \mid g_i\in\Gamma(U_i,\mathcal{F})\right\},\\
C^1(\mathcal{U},\mathcal{F})&:=&\left\{\phi\in\Gamma(U_{\infty}\cap U_0,\mathcal{F})\right\}.
\end{eqnarray*}

Define the coboundary operator $\delta:\mathcal{C}^0(\mathcal{U},\mathcal{F})\rightarrow \mathcal{C}^1(\mathcal{U},\mathcal{F})$ by $\delta g = g_{\infty}-g_0$.
The closed de Rham cocycles are the set
\[Z^1_{\rm dR}\left(\mathcal{U}\right):=\left\{(f,g)\in\mathcal{C}^1(\mathcal{U},\mathcal{O})\times\mathcal{C}^0(\mathcal{U},\Omega^1): df = \delta g\right\}.\]
The de Rham coboundaries are the set
\[B^1_{\rm dR}(\mathcal{S}_m):=\left\{(\delta\kappa, d\kappa)\in Z^1_{\rm dR}\left(\mathcal{U}\right): \kappa\in\mathcal{C}^0(\mathcal{U},\mathcal{O})\right\},\] where $d\kappa=\left(d(\kappa_0),d(\kappa_{\infty})\right)$. 
The de Rham cohomology $H^1_{\rm dR}(\mathcal{S}_m)$ is
\[H^1_{\rm dR}(\mathcal{S}_m)\cong H^1_{\rm dR}(\mathcal{S}_m)(\mathcal{U}):=Z^1_{\rm dR}\left(\mathcal{U}\right)/B^1_{\rm dR}\left(\mathcal{U}\right).\]

There is an injective homomorphism $\lambda:H^0(\mathcal{S}_m,\Omega^1)\rightarrow H^1_{\rm dR}(\mathcal{S}_m)$ denoted informally by $ g \mapsto (0, \mathbf{g})$, where the second coordinate is a tuple $ {\mathbf g} = ( g_{\infty}, g_0)$ defined by $ g_i= g|_{U_i}$.  Define another homomorphism $\gamma: H^1_{\rm dR}(\mathcal{S}_m)\rightarrow H^1(\mathcal{S}_m,\mathcal{O})$ with $(f, \mathbf{g})\mapsto f$.  These create a short exact sequence
\begin{equation} \label{SES}
0\longrightarrow H^0(\mathcal{S}_m,\Omega^1)\mathrel{\mathop{\longrightarrow}^{\lambda}} H^1_{\rm dR}(\mathcal{S}_m)\mathrel{\mathop{\longrightarrow}^{\gamma}}H^1(\mathcal{S}_m,\mathcal{O})\longrightarrow 0.
\end{equation}

\vspace{.1in}

Let $A$ be a basis for $H^1(\mathcal{S}_m, \mathcal{O})$ and $B$ a basis for $H^0(\mathcal{S}_m,\Omega^1)$.  A basis for $\Hdr$ is then given by $\psi(A)\cup\lambda(B)$, where $\psi$ is defined as follows. 
Given $f\in H^1(\mathcal{S}_m,\mathcal{O})$, one can write $df = df_{\infty}+ df_{0}$, where $ df_{i}\in\Gamma(U_i,\Omega^1)$ for $i\in\{0,\infty\}$. 
For convenience, define $\mathbf{df}=\left(df_{\infty},df_{0}\right)$.
Define a section of \eqref{SES} by:
\[\psi:H^1(\mathcal{S}_m,\mathcal{O})\rightarrow \Hdr, \ \psi(f)=\left(f,\mathbf{df}\right).\]  
The image of $\psi$ is a complement in $\Hdr$ to $\lambda(H^0(\mathcal{S}_m,\Omega^1))$. 

\subsubsection{The Frobenius and Verschiebung operators}

The Frobenius $F$ and Verschiebung $V$ act on $\Hdr$ by
\[F(f,\mathbf{g}):=(f^p,(0,0)) \text{ and }V(f,\mathbf{g}):=(0,\mathscr{C}(\mathbf{g}))\]
where $\mathscr{C}$ is the Cartier operator, which acts componentwise on $\mathbf{g}$. The Cartier operator is defined by the properties that it annihilates exact differentials, preserves logarithmic differentials, and is $p^{-1}$-linear.  It follows from the definitions that
\[\text{ker}(F ) = \lambda(H^0(\mathcal{S}_m, \Omega^1 )) = \text{im}(V ).\]

\subsection{The case $m=1$}

When $m=1$, then $q_0=2$, $q=8$, and $g=14$.  The Suzuki curve ${\mathcal S}_1$ has affine equation
\[z^8+z=y^2(y^8+y).\]

The set $\mathcal{E}_1$ consists of the 14 tuples
\begin{eqnarray*}
\mathcal{E}_1&=&\{(0,0,0,0),(0,0,0,1), (0,0,1,0), (0,0,1,1),(0,1,0,0), (0,1,0,1), (0,1,1,0),(1,0,0,0),(1,0,0,1),\\
& & \hspace{.5in} (1,0,1,0), (1,1,0,0), (2,0,0,0), (2,1,0,0), (3,0,0,0)\}.
\end{eqnarray*}

By Lemmas \ref{basisofHomega} and \ref{basisofHO}, $\mathcal{B}_1$ is a basis for $H^0({\mathcal S}_1,\Omega^1)$ and $\mathcal{A}_1$ is a basis for $H^1({\mathcal S}_1,\mathcal{O})$.  
Based on the action of Frobenius and Verschiebung, the following sets make more convenient bases:

\begin{lemma}\label{new bases}
\begin{enumerate}

\item A basis for $H^1({\mathcal S}_1,\mathcal{O})$ is given by the set
\begin{eqnarray*}
A&=&\{f_{(0,0,0,0)},f_{(2,0,0,0)},f_{(0,1,0,0)}+f_{(3,0,0,0)}, f_{(2,1,0,0)}+f_{(0,0,1,0)},f_{(0,0,0,1)}+f_{(1,0,1,0)},\\
& &  f_{(1,0,0,0)},f_{(2,1,0,0)},f_{(1,0,0,1)}, f_{(0,0,1,1)},f_{(1,0,1,0)},
f_{(3,0,0,0)},f_{(1,1,0,0)},f_{(0,1,1,0)}, f_{(0,1,0,1)}\}.
\end{eqnarray*}

\item A basis for $H^0({\mathcal S}_1,\Omega^1)$ is given by the set
\begin{eqnarray*}
B&=&\{ g_{(0,0,0,0)}, g_{(2,0,0,0)}, g_{(0,1,0,0)}+ g_{(3,0,0,0)},  g_{(2,1,0,0)}+ g_{(0,0,1,0)}, g_{(0,0,0,1)}+ g_{(1,0,1,0)},\\
& & g_{(1,0,0,0)}, g_{(2,1,0,0)}, g_{(1,0,0,1)},  g_{(0,0,1,1)}, g_{(1,0,1,0)}, 
g_{(3,0,0,0)}, g_{(1,1,0,0)}, g_{(0,1,1,0)},  g_{(0,1,0,1)}\}.
\end{eqnarray*}
\end{enumerate}
\end{lemma}

\begin{proof}
By Lemma~\ref{basisofHomega} (resp.\ \ref{basisofHO}), these $1$-forms (resp.\ functions) have distinct pole orders at $P_{\infty}$, are therefore linearly independent, and thus form a basis of $H^1({\mathcal S}_1,\mathcal{O})$ (resp. \  $H^0({\mathcal S}_1,\Omega^1)$).
\end{proof}

It is now possible to calculate the action of $F$ and $V$ on $\psi(A)\cup\lambda( B)$, a basis for $H^1_{\rm dR}({\mathcal S}_1)$.  

\subsubsection{The action of Frobenius when $m=1$} 

The action of $F$ is summarized in the right column of Table~\ref{Vtable}.
Note that $F(g)=0$ for $g \in B$ since 
$\text{ker}(F)=\text{im}(V)\cong H^0(\mathcal{S}_1, \Omega^1 )$.
For the action of $F$ on $\psi(f)$ for $f \in A$, note that $F(\psi(f))=(f^2,(0,0))$.
Then
\begin{eqnarray*}
f^2=(f_{(a,b,c,d)})^2&=&(y^{-1-a}z^{1-b}h_1^{1-c}h_2^{1-d})^2\\
&=&(y^{-2})^{1+a}(yh_1+h_2)^{1-b}(z+y^3)^{1-c}(h_1+zy^2)^{1-d}.
\end{eqnarray*}

To do these calculations, 
we simplify $f^2$ and write it as a sum of quotients of monomials in $\left\{y,z, h_1,h_2\right\}$.  These monomials can then be classified as belonging to $\Gamma(U_{0})$ or $\Gamma(U_{\infty})$, or can otherwise be rewritten in terms of the basis for $H^1(\mathcal{S}_1, \mathcal{O})$. It is then possible to use coboundaries to write $(f^2,(0,0))$ in terms of the given basis for $H^1_{\rm dR}(\mathcal{S}_1)$.

\begin{example} To compute that $F(\psi(f_{(0,1,0,1)}))=\lambda(g_{(0,0,0,0)}),$ note first that
\[\left(f_{(0,1,0,1)}\right)^2=y^{-2}(z+y^3)=\frac{z}{y^2}+y.\] Also,
\[d\left(\frac{z}{y^2}\right)=\frac{1}{y^2}dz-2\frac{z}{y^3}dy = dy\textrm{ and }d(y)=dy.\]
Since $y\in\Gamma(U_{\infty},\mathcal{O})$ and $\frac{z}{y^2}\in\Gamma(U_{0},\mathcal{O})$, the pair $\left(\frac{z}{y^2},y\right)$ is in $C^0(\mathcal{U},\mathcal{O})$ and $(\frac{z}{y^2}+y,(dy,dy))$ is a coboundary.  Thus 
\[F\left(\psi\left(f_{(0,1,0,1)}\right)\right)=\left(\frac{z}{y^2}+y, (0,0)\right) + \left(\frac{z}{y^2}+y,(dy,dy)\right) = (0,(dy,dy))=\lambda(dy)=(0,\mathbf{g_{(0,0,0,0)}}).\]
\end{example}

\begin{example}
We compute that $F\left(\psi(f_{(0,0,1,1)})\right)=\psi\left(f_{(0,1,0,1)}\right).$  This is true because 
\[\left(f_{(0,0,1,1)}\right)^2=y^{-2}(yh_1+h_2)=\frac{h_1}{y}+\frac{h_2}{y^2}.\]  Note that $\frac{h_2}{y^2}\in\Gamma(U_{0},\mathcal{O})$, so $(\frac{h_2}{y^2},0)\in C^0(\mathcal{U},\mathcal{O})$, and $d\left(\frac{h_2}{y^2}\right)=\frac{z^4}{y^2}dy$.  So $\left(\frac{h_2}{y^2},(\frac{z^4}{y^2}dy,0)\right)$ is a coboundary.  Also, $d\left(\frac{h_1}{y}\right)=\frac{z^4}{y^2}dy$.  Thus
\[F\left(\psi\left(f_{(0,0,1,1)}\right)\right)=  \left(\frac{h_1}{y}+\frac{h_2}{y^2},(0,0)\right)+\left(\frac{h_2}{y^2},\left(\frac{z^4}{y^2}dy,0\right)\right)= \left(\frac{h_1}{y},\left(\frac{z^4}{y^2}dy,0\right)\right)=\psi\left(f_{(0,1,0,1)}\right).\]
\end{example}

\subsubsection{The action of Verschiebung when $m=1$} 

The action of $V$ is summarized in the middle column of Table~\ref{Vtable}.
In \cite{FGMPW}, the authors calculate the action of the Cartier operator $\mathscr{C}$ (see Table~\ref{thetable}).
This determines the action of $V$ on $\lambda(g)$ for $g \in B$.
It also helps determine the action of $V$ on $\psi(f)$ for $f \in A$.

\begin{example} We compute that $V\left( \psi(f_{(0,1,0,1)})\right)=(0,\mathbf{0})$. Writing $f=f_{(0,1,0,1)}=\frac{h_1}{y}=\frac{z^4}{y}+y^4$, then
\[df=\frac{\partial f}{\partial y}dy+\frac{\partial f}{\partial z}dz=\left(-\frac{z^4}{y^2}+4y^3\right)dy+4\frac{z^3}{y}dz=\frac{z^4}{y^2}dy.\]  Considering the pole orders of $y$, $z$, and $dy$, define $df=df_{0}\in\Omega_0$ and $df_{\infty}=0$, so ${\bf{df}}=(0,df)$.  Thus 
$\mathscr{C}(df)=\frac{z^2}{y}\mathscr{C}(dy)=0$.  Thus $\mathscr{C}({\mathbf {df} })=(0,0)={\mathbf 0}$ and $V( \psi(f_{(0,1,0,1)}))=(0,\mathbf{0})$.
\end{example}

\begin{example} We compute that $V\left( \psi(f_{(2,1,0,0)})\right)=(0,{\mathbf g_{(0,1,0,0)}})$.  This is because 
\[f_{(2,1,0,0)}=\frac{h_1h_2}{y^3},\] so
\[df=y^{-3}d(h_1h_2)+y^{-4}h_1h_2dy=y^{-3}h_1d(h_2)+y^{-3}h_2d(h_1)+y^{-4}h_1h_2dy.\]
Then 
\[d(h_1)=d(z^{4}+y^{5})=y^4dy \textrm{ and }d(h_2)=d(z^4y+h_1^4)=z^4dy,\]
so \begin{align*}
df&=y^{-3}z^4h_1dy+yh_2dy+y^{-4}h_1h_2\\
&=y^{-3}h_1(h_1+y^5)dy+yh_2dy+y^{-4}h_1h_2\\
&=\frac{h_1^2}{y^3}dy+\frac{h_1h_2}{y^4}dy+y^2h_1dy+yh_2dy,\\
\end{align*}
using the fact that $z^4=h_1+y^5$.
Considering the orders of the poles, 
define $df_0=\frac{h_1^2}{y^3}dy+\frac{h_1h_2}{y^4}dy\in\Omega_0$ and $df_{\infty}=y^2h_1dy+yh_2dy\in\Omega_{\infty}$.  
Using Table~\ref{thetable} and the fact that $h_1^2=z+y^3$, then 
\begin{align*}
\mathscr{C}(df_{\infty})&=y\mathscr{C}(h_1dy)+\mathscr{C}(yh_2dy)\\
&=y^3dy+h_1^2dy=(y^3+z+y^3)dy=zdy.
 \end{align*}
 Thus $V\left(\psi(f_{(2,1,0,0)})\right)=(0,{\mathbf g_{(0,1,0,0)}})$.
\end{example}

The actions of $F$ and $V$ are summarized in Table~\ref{Vtable}.  
%

\begin{table}[ht]
\caption{Cartier Operator on $H^0({\mathcal S}_1, \Omega^1)$}\label{thetable}
\noindent \[
\begin{array}{|l|l|}
\hline
f&\mathscr{C}(f dy)\\
\hline
\hline
1&0\\
\hline
y&dy\\
\hline
z&y^{q_0/2}\, dy\\
\hline
h_1&y^{q_0}\, dy\\
\hline
h_2&\left((yh_1)^{q_0/2}+h_2\right)\, dy\\
\hline
yz&h_1^{q_0/2}\,dy\\
\hline
yh_1&\left((yh_1)^{q_0/2}+h_2\right)\, dy\\
\hline
zh_1&(yh_2)^{q_0/2}\, dy\\
\hline
zh_2&(h_1h_2)^{q_0/2}\, dy\\
\hline
h_1h_2&\left(h_1+zy^{q_0}\right)\, dy\\
\hline
yzh_1&\left(y^{q_0/2}z+(h_1h_2)^{q_0/2}\right)\, dy\\
\hline
yzh_2&\left(zh_1^{q_0/2}+y^{q_0/2+1}h_2^{q_0/2}\right)\, dy\\
\hline
zh_1h_2&\left(zy^{q_0/2}h_2^{q_0/2}+h_1^{q_0/2+1}\right)\, dy\\
\hline
yh_1h_2&\left((yh_1)^{q_0/2}z+h_2^{q_0/2}z\right)\, dy\\
\hline
yzh_1h_2&\left(y^{q_0/2}h_2+zh_1^{q_0/2}h_2^{q_02}\right)\, dy\\
\hline
\end{array}\]
\end{table}

\begin{table}[ht]
\caption{Action of Verschiebung and Frobenius on $H^1_{\rm dR}({\mathcal S}_1)$}\label{Vtable}
\noindent \[
\begin{array}{|l|l|l|}
\hline
(f,{\mathbf g})&V(f,{\mathbf g}) &F(f,{\mathbf g}) \\
\hline
\hline
(0, {\mathbf g_{(0,0,0,0)}})&(0,\mathbf{0})& (0, \mathbf{0})\\
(0, {\mathbf g_{(2,0,0,0)}})&(0,\mathbf{0})& (0, \mathbf{0})\\
 (0,{\mathbf g_{(0,1,0,0)}+\mathbf g_{(3,0,0,0)}})&(0,\mathbf{0})& (0, \mathbf{0})\\
 (0,{\mathbf g_{(2,1,0,0)}+\mathbf g_{(0,0,1,0)}})&(0,\mathbf{0})& (0, \mathbf{0})\\
\hline
(0,{\mathbf g_{(0,0,0,1)}+ \mathbf g_{(1,0,1,0)}})&(0,\mathbf{0})& (0, \mathbf{0})\\
\hline
(0,{\mathbf g_{(1,0,0,0)}})&(0, {\mathbf g_{(0,0,0,0)}}) & (0, \mathbf{0})\\
(0,{\mathbf g_{(0,0,1,0)}})&(0, {\mathbf g_{(2,0,0,0)}}) & (0, \mathbf{0})\\
(0,{\mathbf g_{(1,0,0,1)}})&(0, {\mathbf g_{(0,1,0,0)}+\mathbf g_{(3,0,0,0)}}) & (0, \mathbf{0})\\
(0,{\mathbf g_{(0,0,1,1)}})&(0,{\mathbf g_{(2,1,0,0)}+\mathbf g_{(0,0,1,0)}}) & (0, \mathbf{0})\\
\hline
(0,\mathbf g_{(1,0,1,0)})&(0,\mathbf g_{(0,0,0,1)}+\mathbf g_{(1,0,1,0)}) & (0, \mathbf{0})\\
\hline
(0,\mathbf g_{(0,1,0,0)})& (0,\mathbf g_{(1,0,0,0)})& (0, \mathbf{0})\\
(0,\mathbf g_{(1,1,0,0)})& (0,\mathbf g_{(0,0,1,0)})& (0, \mathbf{0})\\
(0,\mathbf g_{(0,1,1,0)})& (0,\mathbf g_{(1,0,0,1)})& (0, \mathbf{0})\\
(0,\mathbf g_{(0,1,0,1)})&(0,\mathbf g_{(0,0,1,1)}) & (0, \mathbf{0})\\
\hline
\psi(f_{(0,1,0,1)}) & (0,\mathbf{0})& \left(0,{\mathbf g_{(0,0,0,0)}} \right)\\
\psi(f_{(0,1,1,0)}) & (0,\mathbf{0})&\left(0,  {\mathbf g_{(2,0,0,0)}}\right)\\
\psi(f_{(1,1,0,0)}) &(0,\mathbf{0})& \left(0,  {\mathbf g_{(0,1,0,0)}+\mathbf g_{(3,0,0,0)}}\right)\\
\psi(f_{(0,1,0,0)}+f_{(3,0,0,0)}) &(0,\mathbf{0})& \left(0,{\mathbf g_{(2,1,0,0)}+\mathbf g_{(0,0,1,0)}}\right)\\
\hline
\psi(f_{(0,0,0,1)}+f_{(1,0,1,0)}) & (0,\mathbf{0}) &\left(0,{\mathbf g_{(0,0,0,1)}+\mathbf g_{(1,0,1,0)}}\right)\\
\hline
\psi(f_{(0,0,1,1)})&(0,\mathbf{0}) & \psi(f_{(0,1,0,1)})\\
\psi(f_{(1,0,0,1)}) & (0,\mathbf{0})&\psi(f_{(0,1,1,0)})\\
\psi(f_{(2,1,0,0)}+f_{(0,0,1,0)})&(0,\mathbf{0})&\psi(f_{(1,1,0,0)})\\
\psi(f_{(1,0,0,0)}) &(0,\mathbf{0})& \psi(f_{(0,1,1,0)})\\
\hline
\psi(f_{(1,0,1,0)}) &(0,\mathbf g_{(1,0,1,0)})& \psi(f_{(0,0,0,1)}+f_{(1,0,1,0)})\\
\hline
\psi(f_{(2,1,0,0)})&(0,\mathbf g_{(0,1,0,0)})&\psi(f_{(0,0,1,1)}) \\
\psi(f_{(3,0,0,0)})&(0,\mathbf g_{(1,1,0,0)})&\psi(f_{(1,0,0,1)})\\
\psi(f_{(2,0,0,0)})&(0,\mathbf g_{(0,1,1,0)})&\psi(f_{(2,1,0,0)}+f_{(0,0,1,0)})\\
\psi(f_{(0,0,0,0)})&(0,\mathbf g_{(0,1,0,1)})&\psi(f_{(1,0,0,0)}) \\
\hline

\end{array}\]
\end{table}

To conclude, we use the tables to give an explicit proof of Proposition \ref{m=1}.

\begin{proposition}
When $m=1$, then 
the mod $2$ Dieudonn\'e module of ${\mathcal S}_1$ is
\[D_1 \simeq \EE/\EE(F^2+V^2) \oplus (\EE/\EE(F^3+V^3))^4.\]
\end{proposition}

\begin{proof}
As an $\EE$-module, $D_1$ is isomorphic to $H^1_{\rm dR}({\mathcal S}_1)$.  
From Table~\ref{Vtable}, $H^1_{\rm dR}({\mathcal S}_1)$ has a summand of rank 4 generated by 
$X_1=\psi(f_{(1,0,1,0)})$ with relation $(F^2+V^2)X_1=0$.  There are 4 summands of rank 6 generated by 
$X_2=\psi(f_{(2,1,0,0)})$, $X_3=\psi(f_{(2,0,0,0)})$, $X_4=\psi(f_{(3,0,0,0)})$, and $X_5=\psi(f_{(0,0,0,0)})$ with the relations $(F^3+V^3)X_i=0$.  
This yields the $\mathbb{E}$-module structure $\mathbb{E}/\EE(F^2+V^2)\oplus \left(\mathbb{E}/\EE(F^3+V^3)\right)^4$. 
\end{proof}

Note that the trivial eigenspace $D_{1,0}$ appears as the summand $\mathbb{E}/(F^2+V^2)$.  
It is spanned by \[\{\psi(f_{(1,0,1,0)}), \psi(f_{(0,0,0,1)}+f_{(1,0,1,0)}), (0,\mathbf g_{(1,0,1,0)}), (0,\mathbf g_{(0,0,0,1)}+\mathbf g_{(1,0,1,0)})\}.\]

\bibliographystyle{amsplain}
\providecommand{\bysame}{\leavevmode\hbox to3em{\hrulefill}\thinspace}
\providecommand{\MR}{\relax\ifhmode\unskip\space\fi MR }
\providecommand{\MRhref}[2]{%
  \href{http://www.ams.org/mathscinet-getitem?mr=#1}{#2}
}
\providecommand{\href}[2]{#2}

\end{document}